\documentclass[reqno,11pt]{amsart}
\pdfoutput1

\usepackage{amssymb,color,hyperref,mathrsfs,stmaryrd}
\usepackage{amsmath}

\usepackage[usenames,dvipsnames]{xcolor}

\usepackage[curve,matrix,arrow]{xy}

\usepackage{xcolor}
\usepackage[normalem]{ulem}

\numberwithin{equation}{section}
\newtheorem{theorem}[equation]{Theorem}
\newtheorem{definition}[equation]{Definition}

\newtheorem{lemma}[equation]{Lemma}

\newtheorem{cor}[equation]{Corollary}

\theoremstyle{definition}

\newcommand{\bN}{\mathbb{N}}
\newcommand{\bC}{\mathbb{C}}

\newcommand{\F}{\mathcal{F}}
\newcommand{\E}{\mathcal{E}}

\renewcommand{\L}{\mathcal{L}}
\newcommand{\mD}{\mathcal{D}}

\newcommand{\M}{\mathcal{M}}
\newcommand{\N}{\mathcal{N}}

\newcommand{\D}{\mathbf{D}}

\renewcommand{\H}{\mathcal{H}}

\newcommand{\Hom}{\operatorname{Hom}}
\newcommand{\Aut}{\operatorname{Aut}}
\newcommand{\Inn}{\operatorname{Inn}}
\newcommand{\Iso}{\operatorname{Iso}}

\newcommand{\id}{\operatorname{id}}

\newcommand{\W}{\mathbf{W}}

\newcommand{\hyp}{\mathfrak{hyp}}

\def \<{\langle }
\def \>{\rangle }
\newcommand{\subn}{\unlhd\!\unlhd\;}

\renewcommand{\phi}{\varphi}

\title[Normalizers and centralizers of $p$-subgroups in normal subsystems of fusion systems]{Normalizers and centralizers of $p$-subgroups in normal subsystems of fusion systems}

\author[E.~Henke]{Ellen Henke}
\address{Institut f{\"u}r Algebra, Fakult{\"a}t Mathematik, Technische Universit{\"a}t Dresden, 01062 Dresden, Germany}
\email{ellen.henke@tu-dresden.de}

\begin{document}

\begin{abstract}
Suppose $\E$ is a normal subsystem of a saturated fusion system $\F$ over $S$. If $X\leq S$ is fully $\F$-normalized, then Aschbacher defined a normal subsystem $N_\E(X)$ of $N_\F(X)$. In this short note we revisit and generalize this result using the theory of localities. Our more general approach leads in particular to a normal subsystem $C_\E(X)$ of $C_\F(X)$ for every $X\leq S$ which is fully $\F$-centralized.
\end{abstract}

\maketitle





\section{Introduction}

Let $\F$ be a saturated fusion system over a $p$-group $S$, and let $\E$ be a normal subsystem of $\F$. If $X\leq S$ is fully normalized in $\F$, Aschbacher \cite[8.24]{Aschbacher:2011} introduced a normal subsystem $N_\E(X)$ of $N_\F(X)$, which should be thought of as the normalizer of $X$ in $\E$ as the notation suggests. The purpose of this paper is to review and generalize this result using the theory of localities and to make transparent how $N_\E(X)$ can be realized inside of a subcentric locality over $\F$. For an introduction to localities, the reader is referred to \cite{Chermak:2015} or to the summary in Sections~4.1, 4.2, and 4.4  of \cite{Henke:Regular}. A \textit{subcentric locality} over $\F$ is a locality $(\L,\Delta,S)$ such that $\F=\F_S(\L)$, $\Delta=\F^s$ and, for every $P\in\Delta$, the group $N_\L(P)$ is of characteristic $p$. Here a finite group $G$ is said to be \emph{of characteristic $p$} if $C_G(O_p(G))\leq O_p(G)$. For the definition of the set $\F^s$ of subcentric subgroups see \cite[Definition~1]{Henke:2015}. 

\smallskip

By \cite[Theorem~A]{Henke:2015}, there exists always a subcentric locality over $\F$. Moreover, if $(\L,\Delta,S)$ is a subcentric locality over $\F$, then it is shown in \cite[Theorem~A]{Chermak/Henke} that there exists a unique partial normal subgroup $\N$ of $\L$ with $\N\cap S=T$ and $\E=\F_T(\N)$. These results lead to a very natural way of showing that $N_\E(X)$ is normal in $N_\F(X)$ if $X\leq S$ is fully normalized in $\F$. More generally, if $X\leq S$ and $K\leq \Aut(X)$, we obtain a similar result for $K$-normalizers provided $X$ is fully $K$-normalized and $K\subn K\Inn(X)$. In particular, this leads to a normal subsystem $C_\E(X)$ of $C_\F(X)$ if $X\leq S$ is fully $\F$-centralized. The reader is referred to Theorem~\ref{T:NEKX} and Corollary~\ref{C:NEXCEX} for the precise statements of our results.

\smallskip

Throughout, $p$ is assumed to be a prime. The reader is referred to \cite[Sections~I.1-I.7]{Aschbacher/Kessar/Oliver:2011} for an introduction to the theory of fusion systems, but we will nevertheless recall some of the definitions below. We adapt the notation and terminology from this reference except that we will conjugate from the right and write homomorphisms on the right hand side of the argument. In particular, if $G$ is a group and $g\in G$, we denote by $c_g$ the map from $G$ to $G$ with $xc_g:=x^g:=g^{-1}xg$.

\section{Localities attached to $p$-local subsystems}

\subsection{$K$-normalizers in fusion systems}

Let $\F$ be a saturated fusion system over a $p$-group $S$, let $X\leq S$ and let $K\leq \Aut(X)$. The $K$-normalizer $N_\F^K(X)$ is then defined and a fusion system over 
\[N_S^K(X):=\{s\in N_S(X)\colon c_s|_X\in K\}.\]
Namely, an $\F$-morphism between subgroups of $N_S^K(X)$ is a morphism in $N_\F^K(X)$ if and only if it extends to an $\F$-morphism that acts on $X$ as an element of $K$. The subgroup $X$ is called \emph{fully $K$-normalized} if 
\[|N_S^K(X)|\geq |N_S^{K^\phi}(X\phi)|\]
for every $\phi\in\Hom_\F(X,S)$, where $K^\phi:=\phi^{-1}K\phi\leq \Aut(X\phi)$. If $X$ is fully $K$-normalized in $\F$, then $N_\F^K(X)$ is saturated by \cite[Theorem~I.5.5]{Aschbacher/Kessar/Oliver:2011}. The most important cases of this $K$-normalizer are the normalizer $N_\F(X):=N_\F^{\Aut(X)}(X)=N_\F^{\Aut_\F(X)}(X)$ and the centralizer $C_\F(X):=N_\F^{\{\id\}}(X)$ of $X$.

\subsection{Restriction} \label{SS:Restrictions} Before we turn attention to $K$-normalizers in subcentric localities, we first summarize a very general construction here. Let $(\L,\Delta,S)$ always be a locality, let $\H$ be a partial subgroup of $\L$ and assume $\Gamma$ is a set of subgroups of $R:=S\cap\H$ such that $\Gamma$ is closed under passing to  $\H$-conjugates and overgroups in $R$. Suppose furthermore that there exists some subgroup $X\leq S$ such that the following properties hold:
\begin{itemize}
 \item [(Q1)] $\<P,X\>\in\Delta$ for all $P\in\Gamma$;
 \item [(Q2)] $N_\H(P_1,P_2)\subseteq N_\L(\<P_1,X\>,\<P_2,X\>)$ for all $P_1,P_2\in\Gamma$.  
\end{itemize}
Set
\[\H|_\Gamma:=\{f\in\H\colon S_f\cap R\in\Gamma\}.\]
Then $\H|_\Gamma$ becomes naturally a partial group as follows: Let $\D_0$ be the set of words $(g_1,\dots,g_n)\in\W(\H)$ for which there exist $P_0,P_1,\dots,P_n\in\Gamma$ with $P_{i-1}^{g_i}=P_i$ for all $i=1,\dots,n$. We mean here to allow the case $n=0$ so that the empty word is an element of $\D_0$. It is shown in \cite[Lemma~9.6]{Henke:2015} that the set $\H|_\Gamma$ together with the restriction of the inversion map on $\L$ to $\H|_\Gamma$ and with the product $\Pi_0:=\Pi|_{\D_0}\colon \D_0\rightarrow \H|_\Gamma$ forms a partial group. If $R=S\cap\H$ is a maximal $p$-subgroup of $\H$, then $(\H|_\Gamma,\Gamma,R)$ is a locality by \cite[Lemma~9.8]{Henke:2015}.

\subsection{$K$-normalizers in localities}

Suppose $(\L,\Delta,S)$ is a subcentric locality over a saturated fusion system $\F$, $X\leq S$ and $K\leq \Aut(X)$. Define
\[N_\L^K(X):=\{f\in N_\L(X)\colon c_f|_X\in K\}\] 
and call this the \emph{$K$-normalizer} of $X$ in $\L$. For $\H\subseteq\L$ set
\[N_\H^K(X)=\H\cap N_\L^K(X).\]
Notice that $N_\L^{\Aut(X)}(X)$ equals the normalizer $N_\L(X)$, and $N_\L^{\{\id\}}(X)$ equals the centralizer $C_\L(X)$ (cf. \cite[Definition~4.4]{Henke:Regular}). By \cite[Lemma~9.10]{Henke:2015}, $N_\L^K(X)$ is a partial subgroup of $\L$ and, if $X$ is fully $K$-normalized, then $N_S^K(X)$ is a maximal $p$-subgroup of $N_\L^K(X)$. It is moreover shown in \cite[Lemma~9.12]{Henke:2015} that $N_\F^K(X)=\F_{N_S^K(X)}(N_\L^K(X))$. 

\smallskip

Suppose now $X$ is fully $K$-normalized and consider the properties (Q1) and (Q2) from Subsection~\ref{SS:Restrictions} for $\H:=N_\L^K(X)$ and $\Gamma:=N_\F^K(X)^s$. Clearly (Q2) holds. Moreover, it is shown in \cite[Lemma~3.14]{Henke:2015} that (Q1) holds as well. Thus it makes sense to define
\[\bN_\L^K(X):=N_\L^K(X)|_{N_\F^K(X)^s}.\]
If $X$ is fully $\F$-normalized, then set 
\[\bN_\L(X):=\bN_\L^{\Aut(X)}(X)=N_\L(X)|_{N_\F(X)^s}\]
and if $X$ is fully $\F$-centralized, then set
\[\bC_\L(X):=\bN_\L^{\{\id\}}(X)=C_\L(X)|_{C_\F(X)^s}.\]
Using the construction described in Subsection~\ref{SS:Restrictions}, we regard these sets as partial groups. If $X$ is fully $K$-normalized and $K\cap\Aut_\F(X)$ is subnormal in $\Aut_\F(X)$, then it is shown in \cite[Lemma~9.13]{Henke:2015} that $(\bN_\L^K(X),N_\F^K(X)^s,N_S^K(X))$ is a subcentric locality over $N_\F^K(X)$. In particular, if $X$ is fully normalized, then $(\bN_\L(X),N_\F(X)^s,N_S(X))$ is a subcentric locality over $N_\F(X)$ and, if $X$ is fully centralized, then  $(\bC_\L(X),C_\F(X)^s,C_S(X))$ is a subcentric locality over $C_\F(X)$. In the next lemma we show a generalization of this result. If $G$ is a finite group, $X\leq G$ and $K\leq \Aut(X)$, we set $N_G^K(X):=\{g\in N_G(X)\colon c_g|_X\in K\}$.

\begin{lemma}\label{L:bNKXSubcentric}
Suppose $(\L,\Delta,S)$ is a subcentric locality over a saturated fusion system $\F$,  $X\leq S$ and $K\leq \Aut(X)$ with $K\subn K\Inn(X)$. Then $(\bN_\L^K(X),N_\F^K(X)^s,N_S^K(X))$ is a subcentric locality over $N_\F^K(X)$
\end{lemma}

\begin{proof}
Set $\L_0:=\bN_\L^K(X)$ for short and let $P\in N_\F^K(X)^s$. The triple $(\L_0,N_\F^K(X)^s,N_S^K(X))$ is a locality over $N_\F^K(X)$ by \cite[Lemma~9.13]{Henke:2015}. Hence, it remains only to show that $N_{\L_0}(P)$ is of characteristic $p$, where the normalizer is a priori formed in $\L_0$, but equals $N_\L(P)\cap \L_0$ by \cite[Lemma~9.7]{Henke:2015}. It follows thus from the definition of $\L_0$ that $N_{\L_0}(P)=N_\L(P)\cap \L_0=N_\L(P)\cap N_\L^K(X)$. By \cite[Lemma~3.14]{Henke:2015}, we have $PX\in\F^s=\Delta$ and thus $G:=N_\L(PX)$ is a group of characteristic $p$. Notice that $N_{\L_0}(P)=N_\L(P)\cap N_\L^K(X)=N_{N_G^K(X)}(P)$. It follows from Lemma~\ref{L:KnormalizerInGroup}(b) below that $N_G^K(X)$ is of characteristic $p$. Hence, $N_{\L_0}(P)=N_{N_G^K(X)}(P)$ is of characteristic $p$ by \cite[Lemma~1.2(c)]{MS:2012b} as required. 
\end{proof}

\begin{lemma}\label{L:KnormalizerInGroup}
 Let $G$ be a finite group of characteristic $p$ and $X$ be a $p$-subgroup of $G$.
\begin{itemize}
 \item [(a)] If $C_G(X)\leq H\leq N_G(X)$ such that $H\subn HX$, then $H$ is of characteristic $p$.
 \item [(b)] If $K\leq \Aut(X)$ with $K\subn K\Inn(X)$, then $N_G^K(X)$ is of characteristic $p$. 
\end{itemize} 
\end{lemma}

\begin{proof}
\textbf{(a)} By \cite[Lemma~1.2(c)]{MS:2012b}, $N_G(X)$ is of characteristic $p$. Moreover, by \cite[Lemma~1.2(a)]{MS:2012b} every subnormal subgroup of a group of characteristic $p$ is also of characteristic $p$. In particular, $C_G(X)\unlhd N_G(X)$ is of characteristic $p$. Moreover, as $H\subn HX$, $H$ is of characteristic $p$ if $HX$ is of characteristic $p$. Replacing $H$ by $HX$ we may thus assume that $X\leq H$. Set now $Y:=XO_p(C_H(X))$. Then $Y$ is a normal $p$-subgroup of $H$ and thus contained in $O_p(H)$. Moreover, $C_H(Y)=C_{C_H(X)}(O_p(C_H(X)))\leq O_p(C_H(X))\leq Y$ as $C_H(X)=C_G(X)$ has characteristic $p$. This shows that $H$ is of characteristic $p$.

\smallskip

\textbf{(b)} Let $K=K_0\unlhd K_1\unlhd\cdots\unlhd K_n=K\Inn(X)$ be a subnormal series of $K$ in $K\Inn(X)$. Observe that
\[N_G^K(X)=N_G^{K_0}(X)\unlhd N_G^{K_1}(X)\unlhd\cdots \unlhd N_G^{K_n}(X)\]
and 
\begin{eqnarray*}
 N_G^{K_n}(X)&=&\{g\in N_G(X)\colon c_g|_X\in K\Inn(X)\}=\{g\in N_G(X)\colon \exists x\in X\mbox{ such that }c_{gx^{-1}}|_X\in K\}\\
&=&\{g\in N_G(X)\colon \exists x\in X\mbox{ such that }gx^{-1}\in N_G^K(X)\}=N_G^K(X)X.
\end{eqnarray*}
Hence, $N_G^K(X)\subn N_G^K(X)X$. Notice that $C_G(X)\leq N_G^K(X)$ as $K$ contains the identity on $X$. Hence, (b) follows from (a) applied with $N_G^K(X)$ in place of $H$. 
\end{proof}


\section{$K$-Normalizers of $p$-subgroups in normal subsystems}

\textbf{Throughout this section, let $\F$ be a saturated fusion system over $S$, and let $\E$ be a normal subsystem of $\F$ over $T\leq S$. Moreover, let $X\leq S$ and $K\leq \Aut(X)$.}

\smallskip

In this subsection we will define $N_\E^K(X)$ if $X$ is fully $K$-normalized. Moreover, if $K\subn K\Inn(X)$, then we show that $N_\E^K(X)$ is normal in $N_\F^K(X)$ and see furthermore how $N_\E^K(X)$ can be realized inside of a subcentric locality over $\F$. 

\smallskip

Before we give the definition of $N_\E^K(X)$ let us briefly recall a few definitions. A subsystem $\mD$ of $\F$ over $R$ is said to have \emph{$p$-power index} in $\F$ if $\hyp(\F):=\<[P,O^p(\Aut_\F(P))]\colon P\leq S\>\leq R$ and $O^p(\Aut_\F(P))\leq \Aut_{\mD}(P)$ for every $P\leq R$. Given a subgroup $R$ of $S$ with $\hyp(\F)\leq R$, it turns out that there is a unique saturated subsystem $\mD$ of $\F$ over $R$ of $p$-power index; moreover, such a subsystem $\mD$ is normal in $\F$ if and only if $R$ is strongly closed; see \cite[Theorem~I.7.4]{Aschbacher/Kessar/Oliver:2011}. The unique saturated subsystem of $\F$ over $\hyp(\F)$ is denoted by $O^p(\F)$. 

\smallskip

We will use that there is a notion of a product $\E X$. Such a product was first introduced by Aschbacher \cite[Chapter~8]{Aschbacher:2011} with a simplified construction given in \cite{Henke:2013}. We argue here based on the latter construction. By \cite[Theorem~1]{Henke:2013},  $\E X$ is saturated and $O^p(\E X)=O^p(\E)$. In particular, $\hyp(\E X)=\hyp(\E)\leq T$. The product $\E X$ is more precisely denoted as $(\E X)_\F$, since it is defined as a certain subsystem of $\F$ which depends on $\F$ (cf. \cite[Example~7.4]{Henke:2013}).

\begin{definition}
\begin{itemize}
\item If $X$ is fully $K$-normalized in $\E X$, then define $N_\E^K(X)$ to be the unique saturated subsystem of $N_{\E X}^K(X)$ over $N_T^K(X):=T\cap N_S^K(X)$ of $p$-power index. 
\item If $X$ is fully normalized in $\E X$, then set $N_\E(X):=N_\E^{\Aut(X)}(X)$. 
\item If $X$ is fully centralized in $\E X$, then set $C_\E(X):=N_\E^{\{\id\}}(X)$. 
\end{itemize}
\end{definition}

Let us explain why the above definition makes sense. If $X$ is fully $K$-normalized in $\E X$, then $N_{\E
X}^K(X)$ is saturated by \cite[Theorem~I.5.5]{Aschbacher/Kessar/Oliver:2011}. Moreover, $\hyp(N_{\E X}^K(X))\leq \hyp(\E X)\leq T$ and
thus $\hyp(N_{\E X}^K(X))\leq N_T^K(X)$ where $N_T^K(X)$ is strongly closed in $N_{\E
X}^K(X)$. So by \cite[Theorem~I.7.4]{Aschbacher/Kessar/Oliver:2011}, there exists a
unique subsystem of $N_{\E X}^K(X)$ over $N_T^K(X)$ of $p$-power index, which is then normal in $N_{\E X}^K(X)$. 

\smallskip

It should be noted that the definition of $N_\E^K(X)$ depends on the
fusion system $\F$, since $\E X=(\E X)_\F$ depends on $\F$ (see \cite[p.7]{Henke/Lynd} for an example). 

\smallskip

If $X$ is fully normalized, if follows from \cite[Theorem~1]{Henke:2013} that $N_\E(X)$ equals $O^p(N_{\E X}(X))N_T(X)$, where the product is formed inside of $N_{\E X}(X)$. Hence, our definition of $N_\E(X)$ agrees with the one given by Aschbacher \cite[8.24]{Aschbacher:2011}. We prove next that $N_\E^K(X)$ is well-defined if $X$ is fully $K$-normalized in $\F$.

\begin{lemma}\label{L:FullyKNormalized}
 If $X$ is fully $K$-normalized in $\F$, then $X$ is also fully $K$-normalized in $\E X$.
\end{lemma}

\begin{proof}
Let $\phi\in\Iso_{\E X}(X,Y)$ such that $Y$ is fully $K^\phi$-normalized in $\E X$. Since $T$ is strongly closed, $(X\cap T)\phi=Y\cap T$. Hence, as $Y\leq TX$, we have $TX=TY$. As $X$ is fully $K$-normalized in $\F$, by \cite[Proposition~I.5.2]{Aschbacher/Kessar/Oliver:2011} (applied with $X$, $Y$ and $\phi^{-1}\in\Iso(Y,X)$ in place of $Q$, $P$ and $\phi$), there exist $\chi\in\Aut_\F^K(X)$ and $\alpha\in\Hom_\F(YN_S^{K^\phi}(Y),S)$ such that $\alpha|_Y=\phi^{-1}\chi$ (and thus $Y\alpha=X$). As $\phi\alpha=\phi(\alpha|_Y)=\chi\in\Aut_\F^K(X)\leq K$, we have $K^{\phi\alpha}=K$. This implies $N_S^{K^\phi}(Y)\alpha\leq N_S^{K^{\phi\alpha}}(X)=N_S^K(X)$. If $yt\in N_{TX}^{K^\phi}(Y)=N_{TY}^{K^\phi}(Y)$ with $y\in Y$ and $t\in T$, then $t=y^{-1}(yt)\in YN_S^{K^\phi}(Y)$ and so $t\alpha$ is defined. Hence, $(yt)\alpha=(y\alpha)(t\alpha)\in XT$ as $Y\alpha=X$ and $T$ is strongly closed in $\F$. This shows $N_{TX}^{K^\phi}(Y)\alpha\leq TX\cap N_S^K(X)=N_{TX}^K(X)$. In particular, $|N_{TX}^{K^\phi}(Y)|\leq |N_{TX}^K(X)|$. As $Y$ is fully $K^\phi$-normalized in $\E X$, it follows that $X$ is fully $K$-normalized in $\E X$. 
\end{proof}

\begin{theorem}\label{T:NEKX}
Suppose $X$ is fully $K$-normalized in $\F$ and $K\subn K\Inn(X)$ or $(K\cap \Aut_\F(X))\subn (K\cap \Aut_\F(X))\Inn(X)$. Then the following hold:
\begin{itemize}
\item [(a)] $N_\E^K(X)$ is a normal subsystem of $N_\F^K(X)$ with $N_\E^K(X)\subseteq \E$. 
\item [(b)] Suppose $(\L,\Delta,S)$ is a subcentric locality over $\F$ and $\N$ is a partial normal subgroup of $\L$ with $T=S\cap \N$ and $\E=\F_T(\N)$. Then $\M:=\N\cap \bN_\L^K(X)$ is a partial normal subgroup of $\bN_\L^K(X)$ with $\M\cap S=\M\cap N_S^K(X)=N_T^K(X)$ and $N_\E(X)=\F_{N_T^K(X)}(\M)$.
\end{itemize}
\end{theorem}

\begin{proof}
Notice that we can always replace $K$ by $K\cap \Aut_\F(X)$. Thus, we may assume that $K\subn K\Inn(X)$. We have shown in Lemma~\ref{L:FullyKNormalized} that $X$ is fully $K$ normlized in $\E X$. Hence, $N_\E^K(X)$ is well-defined. 

\smallskip

By \cite[Theorem~A(b)]{Henke:2015}, there exists a subcentric locality $(\L,\Delta,S)$ over $\F$. Moreover, by \cite[Theorem~A]{Chermak/Henke}, there exists a unique partial normal subgroup $\N$ of $\L$ with $T=S\cap\N$ and $\E=\F_T(\N)$. Assume for the remainder of the proof that $\L$ and $\N$ are chosen that way. As in part (b), set
\[\M:=\N\cap \bN_\L^K(X)\mbox{ and }T_0:=N_T^K(X).\] 
Since $\N$ is a partial normal subgroup of $\L$, one checks easily that $\M$ is a partial normal subgroup of $\bN_\L^K(X)$. Observe that $N_S^K(X)\cap\M=S\cap \M=(S\cap \N)\cap \bN_\L^K(X)=T\cap \bN_\L^K(X)=N_T^K(X)=T_0$. By Lemma~\ref{L:bNKXSubcentric}, $(\bN_\L^K(X),N_\F(X)^s,N_S^K(X))$ is a subcentric locality over $N_\F^K(X)$. Therefore, it follows from \cite[Theorem~A]{Chermak/Henke} that $\E_0:=\F_{T_0}(\M)$ is a normal subsystem of $N_\F^K(X)$. Notice moreover that $\E_0\subseteq\E$ as $\M\subseteq\N$. Hence, it remains only to show that $\E_0=N_\E^K(X)$.

\smallskip 

By \cite[Theorem~G(b)]{Chermak/Henke}, we have $\E X=\F_{TX}(\N X)$ and $\E S=\F_S(\N S)$; here $\N X$ and $\N S$ are partial subgroups of $\L$ by \cite[Lemma~3.20]{Chermak/Henke}. In particular, $\F_{N_{T X}^K(X)}(N_{\N X}^K(X))\subseteq N_{\E X}^K(X)$. As $\M\subseteq N_{\N X}^K(X)$, it follows $\E_0\subseteq N_{\E X}^K(X)$. Since $N_\E^K(X)$ is by definition the unique saturated subsystem of $N_{\E X}^K(X)$ over $T_0:=N_T^K(X)$ of $p$-power index, and as $\E_0$ is a subsystem of $N_{\E X}^K(X)$ over $T_0$, it is sufficient to prove that $\E_0$ has $p$-power index in $N_{\E X}^K(X)$. 

\smallskip

As $\E_0$ is normal in $N_\F^K(X)$, using \cite[Proposition~I.6.4]{Aschbacher/Kessar/Oliver:2011}, one sees easily that $\E_0$ is weakly normal in $N_{\E X}^K(X)\subseteq N_\F^K(X)$. Hence, by \cite[Lemma~4.1]{Henke/Lynd:2017}, to show that $\E_0$ has $p$-power index in $N_{\E X}^K(X)$, we only need to argue that $O^p(\Aut_{N_{\E X}^K(X)}(T_0)\leq \Aut_{\E_0}(T_0)$. For the proof of this property let $\phi\in\Aut_{N_{\E X}^K(X)}(T_0)$ be a $p^\prime$-element. It follows from the definition of $\E X$ and $\E S$ given in \cite[Definition~1]{Henke:2013} that $\E X\subseteq \E S$ and thus $\phi$ is a morphism in $N_{\E S}^K(X)$. Moreover, $T_0$ is strongly closed and thus fully normalized in $N_\F^K(X)$ and in $N_{\E S}^K(X)$. Hence, by the extension axiom, $\phi$ extends to $\hat{\phi}\in\Aut_{N_{\E S}^K(X)}(P)$, where $P:=T_0C_{N_S^K(X)}(T_0)\in N_\F^K(X)^c\subseteq N_\F^K(X)^s$. As $\phi$ is a $p^\prime$-automorphism, replacing $\hat{\phi}$ by a suitable power of $\hat{\phi}$ if necessary, we may assume that $\hat{\phi}$ is a $p^\prime$-element. As $P\in N_\F^K(X)^s$, by \cite[Lemma~3.14]{Henke:2015}, we have $PX\in\F^s$. In particular $N_\L(PX)$ and $N_{\N S}(PX)$ are groups. Since $\hat{\phi}$ extends to a morphism in $\E S=\F_S(\N S)$ which acts on $X$ as an element of $K$, using \cite[Lemma~2.3(c)]{Chermak:2015} we can conclude that $\hat{\phi}$ is realized by a $p^\prime$-element in $N_{\N S}(PX)\cap N_\L^K(X)$. By \cite[Lemma~6.1(b)]{Henke:2020}, $O^p(N_{\N S}(PX))=O^p(N_\N(PX))$. Hence, $\hat{\phi}=c_n|_P$ for some $n\in N_\N(PX)\cap N_\L^K(X)\cap N_\L(P)$. As $P\in N_\F^K(X)^s$, it follows from the definition of $\bN_\L^K(X)$ that $n\in \N\cap \bN_\L^K(X)=\M$ and thus $\phi=c_n|_{T_0}\in \F_{T_0}(\M)=\E_0$. This shows $O^p(\Aut_{N_{\E X}^K(X)}(T_0)\leq \Aut_{\E_0}(T_0)$. As argued before, it follows now that $N_\E^K(X)=\E_0$ and the assertion holds.
\end{proof}

\begin{cor}\label{C:NEXCEX}
\begin{itemize}
 \item [(a)] If $X$ is fully normalized in $\F$, then $N_\E(X)$ is a normal subsystem of $N_\F(X)$ with $N_\E(X)\subseteq \E$. Similarly, if $X$ is fully centralized, then $C_\E(X)$ is normal subsystem of $C_\F(X)$ with $C_\E(X)\subseteq \E$.
\item [(b)] Suppose $(\L,\Delta,S)$ is a subcentric locality over $\F$ and $\N$ is a partial normal subgroup of $\L$ with $T=S\cap \N$ and $\E=\F_T(\N)$. If $X$ is fully $\F$-normalized, then $\M:={\N\cap \bN_\L(X)}$ is a partial normal subgroup of $\bN_\L(X)$ with $\M\cap N_S(X)=\M\cap S=N_T(X)$ and $N_\E(X)=\F_{N_T(X)}(\M)$. Similarly, if $X$ is fully $\F$-centralized, then $\M^\circ:=\N\cap \bC_\L(X)$ is a partial normal subgroup of $\bC_\L(X)$ with $\M^\circ\cap C_S(X)=\M^\circ\cap S=C_T(X)$ and $C_\E(X)=\F_{C_T(X)}(\M^\circ)$. 
\end{itemize}
\end{cor}

\begin{proof}
 This follows from Proposition~\ref{T:NEKX} applied with $K=\Aut(X)$ and $K=\{\id\}$. 
\end{proof}

\bibliographystyle{amsalpha}
\bibliography{repcoh}

\providecommand{\bysame}{\leavevmode\hbox to3em{\hrulefill}\thinspace}
\providecommand{\MR}{\relax\ifhmode\unskip\space\fi MR }
\providecommand{\MRhref}[2]{%
  \href{http://www.ams.org/mathscinet-getitem?mr=#1}{#2}
}
\providecommand{\href}[2]{#2}
\begin{thebibliography}{AKO11}

\bibitem[AKO11]{Aschbacher/Kessar/Oliver:2011}
M.~Aschbacher, R.~Kessar, and B.~Oliver, \emph{{Fusion systems in algebra and
  topology}}, London Math.\ Soc.\ Lecture Note Series, vol. 391, Cambridge
  University Press, 2011.

\bibitem[Asc11]{Aschbacher:2011}
M.~Aschbacher, \emph{The generalized {F}itting subsystem of a fusion system},
  Mem. Amer. Math. Soc. \textbf{209} (2011), no.~986, vi+110.

\bibitem[CH17]{Chermak/Henke}
A.~Chermak and E.~Henke, \emph{Fusion systems and localities -- a dictionary},
  arXiv:1706.05343v2 (2017).

\bibitem[Che15]{Chermak:2015}
A.~Chermak, \emph{{Finite Localities I}},
  https://homepages.abdn.ac.uk/d.j.benson/pages/html/archive/chermak.html,
  older version on arXiv:1505.07786v3 (corrected version 2020, first version
  2015).

\bibitem[Hen13]{Henke:2013}
E.~Henke, \emph{Products in fusion systems}, J. Algebra \textbf{376} (2013),
  300--319.

\bibitem[Hen19]{Henke:2015}
\bysame, \emph{Subcentric linking systems}, Trans. Amer. Math. Soc.
  \textbf{371} (2019), no.~5, 3325--3373.

\bibitem[Hen20]{Henke:2020}
\bysame, \emph{Extensions of homomorphisms between localities},
  arXiv:2006.12626 (2020).

\bibitem[Hen21]{Henke:Regular}
\bysame, \emph{Commuting partial normal subgroups and regular localities},
  arXiv:2103.00955v2 (2021).

\bibitem[HL17]{Henke/Lynd:2017}
E.~Henke and J.~Lynd, \emph{{Extensions of the Benson--Solomon fusion
  systems}}, Geometric and Topological Aspects of the Representation Theory of
  Finite Groups, Springer Proceedings in Mathematics \& Statistics, vol. 242,
  2017.

\bibitem[HL18]{Henke/Lynd}
\bysame, \emph{{Fusion systems with a Benson--Solomon component}}, accepted to
  Duke Math. J., arXiv:1806.01938 (preprint 2018).

\bibitem[MS12]{MS:2012b}
U.~Meierfrankenfeld and B.~Stellmacher, \emph{Applications of the {FF}-{M}odule
  {T}heorem and related results}, J. Algebra \textbf{351} (2012), 64--106.

\end{thebibliography}

\end{document}